\documentclass[a4paper,12pt]{article}

\input{mymacros.sty}

\newcommand{\rhobar}{\overline{\rho}}

\newcommand{\perf}{\mathrm{perf}}

\newcommand{\sing}{{\mathrm{sing}}}
\newcommand{\Gmax}{{G_{\mathrm{max}}}}
\newcommand{\Gmaxbar}{{\Gbar_{\mathrm{max}}}}

\newcommand{\even}{{\mathrm{even}}}
\newcommand{\odd}{{\mathrm{odd}}}

\newcommand{\mf}{\operatorname{mf}}

\newcommand{\N}{n}

\title{A note on homological mirror symmetry
for singularities of type D}
\author{Masahiro Futaki and Kazushi Ueda}
\date{}
\begin{document}

\maketitle

%
%

\begin{abstract}
We prove homological mirror symmetry
for Lefschetz fibrations
obtained as disconnected sums of polynomials
of types A or D.
The proof is based on the behavior
of the Fukaya category
under the addition of a polynomial of type D.
\end{abstract}


%
%

\section{Introduction} \label{sc:introduction}

Let $\N$ be a positive integer.
An invertible $\N \times \N$-matrix
$A = (a_{ij})_{i, j = 1}^\N$ with integer components
defines a polynomial $W \in \bC[x_1, \dots, x_\N]$ by
$$
 W = \sum_{i = 1}^\N x_1^{a_{i 1}} \cdots x_\N^{a_{i\N}}.
$$
Note that non-zero coefficients of $W$
can be absorbed by rescaling $x_i$.
A polynomial obtained in this way
is called an {\em invertible polynomial}
if it has an isolated critical point at the origin.
The quotient ring
$R = \bC[x_1, \dots, x_\N] / (W)$ is naturally graded
by the abelian group $L$
generated by $\N + 1$ elements $\vecx_i$ and $\vecc$ with relations
$$
 a_{i 1} \vecx_1 + \dots + a_{i \N} \vecx_\N = \vecc,
  \qquad i = 1, \dots, \N.
$$
The {\em bounded stable derived category} of $R$
introduced by Buchweitz \cite{Buchweitz_MCM}
is the quotient category
$$
 D^b_\sing(R)
 = D^b (\gr R) / D^\perf (\gr R)
$$
of the bounded derived category $D^b(\gr R)$
of finitely-generated $L$-graded $R$-modules
by its full subcategory $D^\perf(\gr R)$
consisting of bounded complexes of projectives	.
This category originates from the theory of
{\em matrix factorizations}
introduced by Eisenbud \cite{Eisenbud_HACI},
and studied by Orlov \cite{Orlov_TCS}
under the name `{\em triangulated category of singularities}'.
This category is not necessarily closed
under direct summands, and
its idempotent completion will be denoted by $D^\pi_\sing(R)$.

The {\em transpose} of the invertible polynomial $W$
is defined by
$$
 W^\ast = \sum_{i = 1}^\N x_1^{a_{1 i}} \cdots x_\N^{a_{\N i}},
$$
which can be perturbed to an exact Lefschetz fibration
with respect to the standard Euclidean K\"{a}hler form on $\bC^n$.
Let $\Fuk W^\ast$ be the directed $A_\infty$-category
defined by Seidel \cite{Seidel_VC, Seidel_PL}
whose set of objects is a distinguished basis of vanishing cycles
and whose spaces of morphisms are Lagrangian intersection Floer complexes.

The following conjecture comes from the combination of
transposition mirror symmetry
by Berglund and H\"{u}bsch \cite{Berglund-Hubsch}
and homological mirror symmetry
by Kontsevich \cite{Kontsevich_HAMS}:

\begin{conjecture} \label{conj:mirror}
For an invertible polynomial $W$,
there is an equivalence
\begin{equation} 
 D^b_\sing(R) \cong D^b \Fuk W^\ast
\end{equation}
of triangulated categories.
\end{conjecture}

Conjecture \ref{conj:mirror} is known to hold
for Brieskorn-Pham singularities
\cite{Futaki-Ueda_Fermat}.
Takahashi and Ebeling
\cite{Takahashi_WPL, Ebeling-Takahashi_SDWHP}
studies Conjecture \ref{conj:mirror}
from the point of view of
the duality of regular systems of weights
by Saito \cite{Saito_DRSW}.

Recall that the polynomials of types $A_n$ and $D_n$ are defined by
$$
 f = x^{n+1} 
$$
and
$$
 f = x^{n-1} + x y^2 
$$
respectively.
We prove the following in this paper:

\begin{theorem} \label{th:main}
One has an equivalence
\begin{equation} \label{eq:pi_equiv}
 D^\pi_\sing(R) \cong D^b \Fuk W^\ast
\end{equation}
of triangulated categories
if $W^\ast$ is a disconnected sum of polynomials
of types $A$ or $D$.
\end{theorem}

The proof is
based on the study of the behavior
of categories on both sides of \eqref{eq:pi_equiv}
under the addition of a polynomial of type $D$.

The organization of this paper is as follows:
In Section \ref{sc:fuk_dn},
we compute the Fukaya category of a Lefschetz fibration
defined by a polynomial of type $D$.
In Section \ref{sc:fuk_dn+g},
we use induction to compute the Fukaya category
of a disconnected sum of type $A$ and type $D$ polynomials.
The bounded stable derived category of
the transpose of a type $D$ singularity is computed in Section \ref{sc:dbsing},
and the behavior of stable derived categories
under disconnected summation of polynomials is studied
in Section \ref{sc:dbsing_sum}.
In Section \ref{sc:group_actions},
we discuss a possible generalization of Conjecture \ref{conj:mirror}
to the case with group actions when $n = 2$.

\section{The Fukaya category of $x^{n-1} + x y^2$}
 \label{sc:fuk_dn}

Let
$$
 f(x, y) = x^{n-1} + x y^2 + y
$$
be a perturbation of a polynomial
of type $D_n$.
The critical points of $f$ are given by
\begin{align*}
\left\{
\begin{aligned}
 x^n &= - \frac{1}{4 (n - 1)}, \\
 y &= - \frac{1}{2 x},
\end{aligned}
\right.
\end{align*}
with critical values
$$
 f(x, y) = - \frac{n}{n-1} \cdot \frac{1}{4 x}.
$$
Consider the diagram
\vspace{7mm}
$$
\begin{psmatrix}[colsep=1.5]
 \bC^{2} & \bC^2 & \bC
\end{psmatrix}
\psset{shortput=nab,arrows=->,labelsep=3pt}
\small
\ncline{1,1}{1,2}_{\varpi}
\ncline{1,2}{1,3}_{\psi}
\ncarc[arcangle=30]{1,1}{1,3}^{\Psi = \psi \circ \varpi}
$$
where
$$
 \varpi(x, y) = (f(x, y), x)
$$
and
$$
 \psi(y_1, y_2) = y_1.
$$
For general $t \in \bC$, the map
$$
 \scE_t \xto{\varpi_t} \scS_t
$$
from $\scE_t = \Psi^{-1}(t)$ to $\scS_t = \psi^{-1}(t)$
is a double cover branching at
$$
 \{ x \in \bC \mid 4 x^n - 4 t x - 1 = 0 \}.
$$
Besides these branch points,
the origin is a distinguished point
with respect to the projection $\varpi_t$,
since one of two points in the fiber $\varpi_t^{-1}(x)$
goes to infinity at $x = 0$.

Now choose a distinguished set of vanishing paths
as the straight line segments from the origin to critical values
as shown in Figure \ref{fg:vp_f_1}.
The trajectories of the branch points of $\varpi_t$
along these paths are shown in Figure \ref{fg:vc_f_1},
which are the images of the vanishing cycles by $\varpi_0$.
By the mutation of the distinguished set of vanishing paths
as in Figure \ref{fg:vp_f_2},
one obtains the vanishing cycles
shown in Figure \ref{fg:vc_f_2}.
By continuing mutations,
one arrives at the distinguished set of vanishing paths
shown in Figure \ref{fg:vp_f_3}.
The images of the corresponding vanishing cycles by $\varpi_0$
are shown in Figure \ref{fg:vc_f_3}.

\begin{figure}
\begin{minipage}{.4 \linewidth}
\centering
\input{vp_f_1.pst}
\caption{A distinguished set of vanishing paths}
\label{fg:vp_f_1}
\end{minipage}
\begin{minipage}{.1 \linewidth}
\ \\
\end{minipage}
\begin{minipage}{.4 \linewidth}
\centering
\input{vc_f_1.pst}
\caption{The image of the vanishing cycles}
\label{fg:vc_f_1}
\end{minipage}
%
\begin{minipage}{.5 \linewidth}
\centering
\input{vp_f_2.pst}
\caption{Vanishing paths after a mutation}
\label{fg:vp_f_2}
\end{minipage}
\begin{minipage}{.5 \linewidth}
\centering
\input{vc_f_2.pst}
\caption{Corresponding vanishing cycles}
\label{fg:vc_f_2}
\end{minipage}
\end{figure}

\begin{figure}
\begin{minipage}{.5 \linewidth}
\centering
\input{vp_f_3.pst}
\caption{Vanishing paths
after a sequence of mutations}
\label{fg:vp_f_3}
\end{minipage}
\begin{minipage}{.5 \linewidth}
\centering
\input{vc_f_3.pst}
\caption{Corresponding vanishing cycles}
\label{fg:vc_f_3}
\end{minipage}
\end{figure}

Note that one can perturb $C_1$ by a Hamiltonian diffeomorphism
$
 \psi: f^{-1}(0) \to f^{-1}(0)
$
so that $\psi(C_1)$ does not intersect with $C_2$.
Now it is easy to see that
there is a quasi-equivalence
$$
 D^b \Fuk f \simto D^b \module \Gamma
$$
from the derived Fukaya category of $f$
to the derived category of finite-dimensional modules
over the path algebra of the quiver
\begin{equation} \label{eq:Dynkin_quiver}
\Gamma =
\left(
\begin{array}{c@{\hskip13mm}c@{\hskip13mm}
c@{\hskip13mm}c@{\hskip13mm}c}
 \Rnode{01}{v_1} & & & \\[5mm]
 &
 \Rnode{12}{v_3} &
 \Rnode{13}{v_4} &
 \Rnode{14}{\cdots} &
 \Rnode{15}{v_n} \\[5mm]
 \Rnode{21}{v_2} & & & &
\end{array}
\psset{nodesep=3pt}
\ncline{<-}{01}{12}
\ncline{<-}{21}{12}
\ncline{<-}{12}{13}
\ncline{<-}{13}{14}
\ncline{<-}{14}{15}
\right)
\end{equation}
such that the vanishing cycle $C_i$ is mapped
to the simple module $S_i$ associated with the $i$-th vertex $v_i$
for $i = 1, \dots, n$.

\section{The Fukaya category of $x^{n-1} + x y^2 + g(z)$}
 \label{sc:fuk_dn+g}

Let
$$
 f(x, y) = x^{n-1} + x y^2 + y
$$
be a perturbation of a polynomial
of type $D_n$
and
$$
 g : \bC^k \to \bC
$$
be an exact symplectic Lefschetz fibration.
Consider the diagram
\vspace{7mm}
$$
\begin{psmatrix}[colsep=1.5]
 \bC^{k+2} & \bC^2 & \bC
\end{psmatrix}
\psset{shortput=nab,arrows=->,labelsep=3pt}
\small
\ncline{1,1}{1,2}_{\varpi}
\ncline{1,2}{1,3}_{\psi}
\ncarc[arcangle=30]{1,1}{1,3}^{\Psi = \psi \circ \varpi}
$$
where
$$
 \varpi(x, y, z) = (f(x, y) + g(z), x)
$$
and
$$
 \psi(y_1, y_2) = y_1.
$$
We write the critical points of $f$ and $g$ as
$$
 \Crit f = \{ (x_i, y_i) \}_{i=1}^n
$$
and
$$
 \Crit g
  = \{ z_j \}_{j=1}^{m},
$$
so that the set of critical points of $\Psi$ is given by
$$
 \Crit \Psi
  = \Crit f \times \Crit g
  = \{ p_{ij} = (x_i, y_i, z_j) \}_{i, j}
$$
with critical values 
$$
 \Psi(p_{ij}) = f(x_i, y_i) + g(z_j).
$$
Assume for simplicity that
the set of critical values of $g$
is the set of $m$-th roots of unity.
Figure \ref{fg:critv_fg} shows the critical values of $f$ and $g$
in the case $n = 4$ and $m = 3$, and
Figure \ref{fg:critv_Psi} shows
the corresponding critical values of $\Psi$.

\begin{figure}
\begin{minipage}{.5 \linewidth}
\centering
\input{critv_fg.pst}
\caption{Critical values of $f$ (outside) and $g$ (inside)}
\label{fg:critv_fg}
\end{minipage}
\begin{minipage}{.5 \linewidth}
\centering
\input{critv_Psi.pst}
\caption{Critical values of $\Psi$}
\label{fg:critv_Psi}
\end{minipage}
\end{figure}

\begin{figure}
\begin{minipage}{.5 \linewidth}
\centering
\input{Psi_vp.pst}
\caption{A distinguished set $(\gamma_{ij})_{ij}$
of vanishing paths for $\Psi$}
\label{fg:Psi_vp}
\end{minipage}
\begin{minipage}{.5 \linewidth}
\centering
\input{s-plane_mp.pst}
\caption{Trajectories of $t - \Crit(g)$
along the vanishing paths $\gamma_{ij}$}
\label{fg:s-plane_mp}
\end{minipage}
%
\begin{minipage}{.5 \linewidth}
\centering
\scalebox{.95}{\input{u-plane_mp1.pst}}
\caption{Matching paths on the $x$-plane}
\label{fg:u-plane_mp1}
\end{minipage}
\begin{minipage}{.5 \linewidth}
\centering
\input{u-plane_mp.pst}
\caption{Matching paths after distortion}
\label{fg:u-plane_mp}
\end{minipage}
\end{figure}

\begin{figure}
\begin{minipage}{.5 \linewidth}
\centering
\input{u-plane_vp1.pst}
\caption{A distinguished set $(\delta_{ij})_{ij}$
of vanishing paths for $\varpi_0$}
\label{fg:u-plane_vp1}
\end{minipage}
\begin{minipage}{.5 \linewidth}
\centering
\input{u-plane_vp2.pst}
\caption{Another distinguished set $(\delta'_{ij})_{ij}$
of vanishing paths for $\varpi_0$}
\label{fg:u-plane_vp2}
\end{minipage}
\end{figure}

For general $t \in \bC$,
the map
$$
 \scE_t \xto{\varpi_t} \scS_t
$$
from $\scE_t = \Psi^{-1}(t)$ to $\scS_t = \psi^{-1}(t)$
is a Lefschetz fibration away from the origin,
and a point $(t, x) \in \scS_t$ is a critical value of $\varpi_t$
if there is a solution to the equations
\begin{align*}
\left\{
\begin{aligned}
 \frac{\partial f}{\partial x}(x, y)
  &= \frac{\partial f}{\partial y}(x, y)
  = \frac{\partial g}{\partial z}(z)
  = 0, \\
 f(x, y, z) &= t.
\end{aligned}
\right.
\end{align*}
This condition is equivalent to
\begin{align*}
\left\{
\begin{aligned}
 & 4 x^n - 4 s x - 1 = 0, \text{ and}\\
 & s = t - g(z_j) \text{ for some } j \in \{ 1, \dots, m \}.
\end{aligned}
\right.
\end{align*}
If we write
$$
 D(s) = \{ x \in \bC \mid 4 x^n - 4 s x - 1 = 0 \},
$$
then the trajectory of $D(t - \Crit(g))$
as $t$ varies along a vanishing path
is a matching path corresponding to a vanishing cycle of $\Psi$.
We choose a distinguished set $(\gamma_{ij})_{ij}$ of vanishing paths
as in Figure \ref{fg:Psi_vp}.
The trajectories of $t - \Crit(g)$ are shown
in Figure \ref{fg:s-plane_mp}, and
the matching paths $\mu_{ij}$ corresponding to vanishing cycles $C_{ij}$
are shown in Figure \ref{fg:u-plane_mp1}.
Figure \ref{fg:u-plane_mp} is obtained from Figure \ref{fg:u-plane_mp1}
by distorting for better legibility.

Let $C_j \subset g^{-1}(0)$ be
the vanishing cycle of $g$ along the straight line segment
from the origin to $g(z_j)$, and
choose a base point $\ast$ on the $x$-plane and
a distinguished set $(\delta_{ij})$ of vanishing paths for $\varpi_0$
as in Figure \ref{fg:u-plane_vp1}.
Since $f(x, y)$ is quadratic in the variable $y$,
the fiber $\varpi_0^{-1}(\ast)$ is a suspension of $g^{-1}(0)$,
and the vanishing cycle $\Delta_{i,j}$ of $\varpi_0$ along $\delta_{ij}$
is a suspension of $C_j$.
If write the Fukaya category of $g^{-1}(0)$ consisting of $\{ C_j \}_{j=1}^m$
and its directed subcategory
as $\scB$ and $\scA$ respectively,
then the Fukaya category $\scB^\sigma_n$ of $\varpi_0^{-1}(\ast)$
consisting of $\Delta_{i,j}$ satisfies
$$
 \hom_{\scB^\sigma_n}(\Delta_{i,j}, \Delta_{i',j'})
  = \hom_{\scB^\sigma}(C_j^\sigma, C_{j'}^\sigma),
$$
where $(\scB^\sigma, \scA^\sigma)$ is the algebraic suspension
of the pair $(\scB, \scA)$
defined by Seidel \cite{Seidel_suspension}.
Let
$$
 S_{ij} = \Cone(\Delta_{i+1,j} \to \Delta_{i,j}), \qquad
  i = 1, \dots, n, \quad
  j = 1, \dots, m
$$
be the object of $D^b(\scA^\sigma_n)$
defined as the cone over the morphism
$e_{i, j} : \Delta_{i+1,j} \to \Delta_{i,j}$
corresponding to $\id_{C_j^\sigma}$ under the isomorphism
$$
 \hom_{\scA^\sigma_n} (\Delta_{i+1,j}, \Delta_{i,j})
  = \hom_{\scB^\sigma}(C_j^\sigma, C_{j}^\sigma).
$$
The following theorem gives a description
of the $A_\infty$-structure on $C_{ij}$ for $i \ne 2$
in terms of the Fukaya category of $g^{-1}(0)$:

\begin{theorem}
[{Seidel
\cite[Proposition 18.21]{Seidel_PL}}]
 \label{th:Seidel}
There is a cohomologically full and faithful functor
from the Fukaya category of $\Psi^{-1}(0)$
consisting of vanishing cycles $C_{ij}$
for $i \ne 2$ to $D^b \scA^\sigma_n$,
such that $C_{1j}$ are mapped to $S_{1j}$ and
$C_{ij}$ for $i \ge 3$ are mapped to $S_{i-1,j}$.
\end{theorem}

Note that this description is completely parallel to
the case of type $A$ singularities
given in \cite{Futaki-Ueda_Fermat}.
One can repeat the same argument
using the distinguished set of vanishing paths
in Figure \ref{fg:u-plane_vp2}
to give an identical description
of the $A_\infty$-structure on $C_{ij}$ for $i \ne 1$.
Moreover,
since the vanishing cycles $C_{1j}$ and $C_{2j}$ can be obtained
as iterated suspensions of the vanishing cycles $C_1$ and $C_2$
in Section \ref{sc:fuk_dn},
one can choose a Hamiltonian diffeomorphism
$\psi : \Psi^{-1}(0) \to \Psi^{-1}(0)$
satisfying $\psi(C_{1j}) \cap C_{2j'} = \emptyset$
for any $j$ and $j'$.
This shows that there are no morphisms
from $C_{1j}$ to $C_{2j'}$
in the cohomology category of the Fukaya category of $\Psi$.

Now one can follow the same argument
as \cite[Section 3]{Futaki-Ueda_Fermat}
to show that the Fukaya category of a disconnected sum
of polynomials of types $A$ and $D$
is quasi-equivalent to the tensor product of the graded categories
associated with individual polynomials;
it is straightforward to check this at the level of cohomology categories,
and higher $A_\infty$-operations vanish for degree reasons.

\section{The stable derived category of $x^{n-1} y + y^2$}
 \label{sc:dbsing}

Let
$$
 W = x^{n - 1} y + y^2.
$$
be the transpose of a polynomial of type $D_n$.
The abelian group
$$
 L = \bZ \vecx \oplus \bZ \vecy \oplus \bZ \vecc
  / ((n-1) \vecx + \vecy - \vecc, 2 \vecy - \vecc)
$$
is isomorphic to $\bZ$, so that the coordinate ring
$$
 R = \bC[x, y] / (W)
$$
is graded as
$$
 \deg(x, y) = (1, n-1).
$$
First recall the following:

\begin{lemma}[{Orlov \cite[Lemma 2.4]{Orlov_DCCSTCS}}]
There is a weak semiorthogonal decomposition
$$
 D^b (\gr R_{\ge 0}) = \langle \scP_{\ge 0}, \scT_0 \rangle
$$
where $\gr R_{\ge 0}$ is the abelian category of
$\bN$-graded $R$-modules,
$\scP_{\ge 0}$ is the subcategory of $D^b( \gr R_{\ge 0})$
generated by $R(i)$ for $i \le 0$, and
$\scT_0$ is equivalent
to $D^b_\sing(R)$.
\end{lemma}
Let $\m = (x, y)$ be the maximal ideal of the origin.
For a graded $R$-module $M$,
another graded $R$-module
obtained by shifting the degree of $M$ by $a \in \bZ$
will be denoted by $M(a)$;
$
 M(a)_\ell = M_{a+\ell}.
$

\begin{lemma}
The graded modules $R/(y)$, $R/(x^{n-1}+y)$ and
$R/\m(i)$ for $i = - n + 3, - n + 4, \dots, -1, 0$
belong to $\scT_0$.  
\end{lemma}
\begin{proof}
The projective resolutions
\begin{align*}
 \dots \to
 \begin{array}{c}
   R(-2n+1) \\ \oplus \\ R(-3n+3)
 \end{array}
  &\xto{
   \begin{pmatrix}
     -y & x^{n-2} y \\
     x & y
   \end{pmatrix}
     }
 \begin{array}{c}
   R(-n) \\ \oplus \\ R(-2n+2)
 \end{array}
 \xto{
   \begin{pmatrix}
     -y & x^{n-2} y \\
     x & y
   \end{pmatrix}
     }
 \begin{array}{c}
   R(-1) \\ \oplus \\ R(-n+1)
 \end{array} \\
 & \qquad \xto{
   \begin{pmatrix}
     x & y
   \end{pmatrix}
     }
  R \to R/\m \to 0,
\end{align*}
\begin{align*}
 \dots \to R(-2n+2)
  \xto{x^{n-1} + y} R(-n+1)
  \xto{y} R
  \to R/(y) \to 0,
\end{align*}
\begin{align*}
 \dots \to R(-2n+2)
  \xto{y} R(-n+1)
  \xto{x^{n-1} + y} R
  \to R/(x^{n-1}+y) \to 0
\end{align*}
show that
\begin{align*}
 \RHom_{D^b \module R}(R/\m, R) &\cong R/\m(-n+2)[-1], \\
 \RHom_{D^b \module R}(R/(y), R) &\cong R/(x^{n-1}+y)(-n+1), \\
 \RHom_{D^b \module R}(R/(x^{n-1}+y), R) &\cong R/(y)(-n+1).
\end{align*}
It follows that
\begin{align*}
 \RHom_{D^b \gr R}(R/(y), R(\ell)) &= 0, \qquad \ell \le 0, \\
 \RHom_{D^b \gr R}(R/(x^{n-1}+y), R(\ell)) &= 0, \qquad \ell \le 0, \\
 \RHom_{D^b \gr R}(R/\m(i), R(\ell)) &= 0, \qquad \ell \le 0,
\end{align*}
for $i = - n + 3, - n + 4, \dots, -1, 0$,
and the lemma is proved.
\end{proof}

Now we have the following:
\begin{lemma}
The graded modules $R/(y)$, $R/(x^{n-1}+y)$ and
$R/\m(i)$ for $i = - n + 3, - n + 4, \dots, -1, 0$
generate $D^b_\sing(R)$ as a triangulated category.
\end{lemma}
\begin{proof}
The exact sequence
$$
 0 \to R \to R/(y) \oplus R/(x^{n-1}+y) \to R / (x^{n-1}, y) \to 0
$$
shows that the module
$
 R / (x^{n-1}, y)
$
can be obtained from $R/(y)$ and $R/(x^{n-1}+y)$
by taking cones up to perfect complexes.
Then the exact sequences
$$
\begin{array}{ccccccccc}
 0 &\to& R/\m(-1) &\to& R/(x^2, y) &\to& R/\m &\to& 0, \\
 0 &\to& R/\m(-2) &\to& R/(x^3, y) &\to& R/(x^2, y) &\to& 0, \\
   &&&& \vdots &&&& \\
 0 &\to& R/\m(-n+2) &\to& R/(x^{n-1}, y) &\to& R/(x^{n-2}, y) &\to& 0, \\
\end{array}
$$
show that
$R/\m(-n+2)$ can be obtained from
$R/(y)$, $R/(x^{n-1}+y)$, $R/\m$, \dots, $R/\m(-n+2)$, and $R/\m(-n + 3)$
by taking cones up to perfect complexes.
Now note that
$$
 0 \to R/(y)(-1) \to R/(y) \to R/\m \to 0
$$
shows that $R/(y)(-1)$ can be obtained from $R/(y)$ and $R/\m$,
and $R/(x^{n-1}+y)(-1)$ can be obtained from $R/(x^{n-1}+y)$ and $R/\m$
in the same way.
Then by repeating the above argument with degree shifted,
one can obtain $R/\m(-n+1)$, $R/\m(-n)$ and so on.

Now consider the exact sequences
$$
 0 \to R(-1) \xto{x} R \to R/(x, y^2)
$$
and
$$
 0 \to R/\m(-n+1) \to R/(x, y^2) \to R/\m \to 0.
$$
These show that
$R/\m(a + n - 1)$ can be obtained from $R/\m(a)$
by taking cones up to perfect complexes.

Alternatively, one can use a result of Orlov
\cite[Theorem 2.5]{Orlov_DCCSTCS}
and the fact that $R$ is Gorenstein with the parameter $2 - n$,
which gives the weak semiorthogonal decomposition
$$
 \scT_0
  = \langle R/\m, R/\m(-1), \dots, R/\m(-n+3), \scD_{n-2} \rangle
$$
where $\scT_0$ is equivalent to $D^b_\sing(R)$ and 
$\scD_{n-2}$ is equivalent to $\qgr(R)$.
Although $R/(y)$ and $R/(x^{n-1}+y)$ do not belong to $\scD_{n-2}$,
suitable cones over $R/(y)$, $R/(x^{n-1}+y)$ and $M_i$ for $i = - n + 3, \dots, 0$
belongs to $\scD_{n-2}$,
and these cones generate $\scD_{n-2} \cong \qgr(R)$.
\end{proof}

By computing the $\Ext$-groups between these generators,
one can show the equivalence
$$
 D^b_\sing(R) \cong D^b \module \Gamma
$$
with the derived category of finite-dimensional representations
of the Dynkin quiver $\Gamma$ in \eqref{eq:Dynkin_quiver},
such that the graded modules
$R / (y)$, $R / (x^{n-1} + y)$ and $R / \frakm(i)$
for $i = 0, -1, \dots, -n+3$ goes to
the simple modules corresponding to the vertices
$v_1$, $v_2$ and $v_{-i+3}$.

\section{The stable derived category of a disconnected sum}
 \label{sc:dbsing_sum}

We discuss a graded analogue of a result of Dyckerhoff
\cite[Theorem 5.1]{Dyckerhoff_CGCMF} in this section.
Let $W$ be an invertible polynomial in $S = \bC[x_1, \dots, x_n]$
graded by $L$.
A {\em graded matrix factorization} is an infinite sequence
$$
 K^\bullet =
  \{ \cdots \to K^i \xto{k^i} K^{i+1} \xto{k^{i+1}} K^{i+2} \to \cdots \}
$$
of morphisms of
$L$-graded free $S$-modules
such that
$
 k^{i+1} \circ k^i = W
$
and
$K^\bullet[2] = K^\bullet(\vecc)$.
A morphism between two graded matrix factorizations
$K^\bullet$ and $H^\bullet$ is a family of morphisms
$
 f^i : K^i \to H^i
$
of $L$-graded modules such that
$
 f^{i+2} = f^i(\vecc).
$
The composition and the differential
on morphisms are defined in just the same way
as unbounded complexes of $L$-graded modules.
The homotopy category
of the differential graded category
$\mf_S^L(W)$ of finitely-generated graded matrix factorizations is equivalent
to the stable category of graded maximal Cohen-Macaulay modules
over $R = S / (W)$
by Eisenbud \cite{Eisenbud_HACI},
which in turn is equivalent to the bounded stable derived category
$D^b_\sing(R)$ by Buchweitz \cite{Buchweitz_MCM}.

Recall that an object $\scE$ in a triangulated category is
a {\em generator} if
$\Hom(\scE, \scF[i]) = 0$ for all $i \in \bZ$ implies $\scF \cong 0$.
Let $\Lbar \subset L$ be a representative
of the finite abelian group $L / \bZ \vecc$ and
$\frakm = (x_1, \dots, x_n)$ be the maximal ideal of $S$
corresponding to the origin.
Since $R$ has an isolated singularity at the origin,
a result of Schoutens
\cite{Schoutens_PDSL},
Murfet \cite[Proposition A.2]{Keller-Murfet-Van_den_Bergh},
Orlov \cite[Proposition 2.7]{Orlov_FC}, or
Dyckerhoff \cite[Corollary 4.3]{Dyckerhoff_CGCMF} shows that
the object
$$
 \scE = \bigoplus_{\ell \in \Lbar} R/\frakm(\ell)
$$
is a generator of $D^b_\sing(R)$.

The graded matrix factorization $K^\bullet$
corresponding to $R / \frakm$ is given by
\begin{align*}
 K^i &=
  \begin{cases}
   \displaystyle{\bigoplus_{j\,:\,\text{even}} \Omega_S^j}
    & \text{$i$ is even}, \\
   \displaystyle{\bigoplus_{j\,:\,\text{odd}} \Omega_S^j}
    & \text{$i$ is odd},
  \end{cases} \\
 k^i &= \iota_\eta + \gamma \wedge \cdot
\end{align*}
with a suitable grading, where
$$
 \eta = \sum_i x_i \partial_i
$$
is the Euler vector field and
$\gamma$ is a one-form such that
$$
 W = \gamma(\eta).
$$
Now assume that $W$ is a disconnected sum
of two invertible polynomials $W_1 \in S_1$ and $W_2 \in S_2$,
and $\Lbar$ is the image of the product
$\Lbar_1 \times \Lbar_2$
of representatives
$\Lbar_1 \subset L_1$ and $\Lbar_2 \subset L_2$
by the natural projection
$$
 L_1 \oplus L_2 \to L = L_1 \oplus L_2 / (\vecc_1 - \vecc_2).
$$
Then the matrix factorization $K^\bullet$ for $W$
is the tensor product of matrix factorizations
$K_1^\bullet$ and $K_2^\bullet$ for $W_1$ and $W_2$;
\begin{align*}
 K^i &=
  \begin{cases}
   \Omega_{S_1}^\even \otimes \Omega_{S_2}^\even
    \oplus \Omega_{S_1}^\odd \otimes \Omega_{S_2}^\odd
    & \text{$i$ is even}, \\
   \Omega_{S_1}^\even \otimes \Omega_{S_2}^\odd
    \oplus \Omega_{S_1}^\odd \otimes \Omega_{S_2}^\even
    & \text{$i$ is odd},
  \end{cases} \\
 k^i &= (\iota_{\eta_1} + \gamma_1 \wedge \cdot)
         + (\iota_{\eta_2} + \gamma_2 \wedge \cdot).
\end{align*}
It follows that one has a quasi-isomorphism
of differential graded algebras
$$
 \hom_{\module S}(K^\bullet, K^\bullet)
  \cong \hom_{\module S_1}(K_1^\bullet, K_1^\bullet)
   \otimes \hom_{\module S_2}(K_2^\bullet, K_2^\bullet),
$$
which induces a quasi-isomorphism
$$
 \End_{\mf_S^L(W)}
  \lb \scE \rb
 \cong
   \End_{\mf_{S_1}^{L_1}(W_1)}
    \lb \scE_1 \rb
  \otimes
   \End_{\mf_{S_2}^{L_2}(W_2)}
    \lb \scE_2 \rb
$$
by passing to $L$-graded pieces.
Since $\scE$ is a compact generator of $\mf_S^L(W)$,
this shows that the tensor product
$
 \mf_{S_1}^{L_1}(W_1) \otimes \mf_{S_2}^{L_2}(W_2)
$
of differential graded categories
is quasi-equivalent to $\mf_S^L(W)$
up to direct summands.

\section{Group actions and crepant resolutions}
 \label{sc:group_actions}

Let $W$ be an invertible polynomial
associated with an $n \times n$ matrix $A = (a_{ij})_{ij}$.
The abelian group $L$ is the group of characters of $K$
defined by
$$
 K =
 \{ (\alpha_1, \dots, \alpha_n) \in (\bCx)^n
       \mid \alpha_1^{a_{11}} \cdots \alpha_n^{a_{1n}}
              = \dots = \alpha_1^{a_{n1}} \cdots \alpha_n^{a_{nn}}
 \}.
$$
The group $\Gmax$ of {\em maximal diagonal symmetries}
is defined as the kernel of the map
$$
\begin{array}{ccc}
 K & \to & \bCx \\
 \vin & & \vin \\
 (\alpha_1, \dots, \alpha_n) & \mapsto &
 \alpha_1^{a_{11}} \cdots \alpha_n^{a_{1n}},
\end{array}
$$
so that
there is an exact sequence
$$
 1 \to \Gmax \to K \to \bCx \to 1.
$$
This exact sequence induces an exact sequence
$$
 1 \to \bZ \to L \to \Gmax^\vee \to 1
$$
of the corresponding character groups,
where
$$
 \Gmax^\vee = \Hom(\Gmax, \bCx)
$$
is non-canonically isomorphic to $\Gmax$.
If we write
$$
 A^{-1} =
 \begin{pmatrix}
    \varphi_1^{(1)} & \varphi_1^{(2)} & \cdots & \varphi_1^{(n)}\\
    \varphi_2^{(1)} & \varphi_2^{(2)} & \cdots & \varphi_2^{(n)}\\
    \vdots & \vdots & \ddots & \vdots \\
    \varphi_n^{(1)} & \varphi_n^{(2)} & \cdots & \varphi_n^{(n)}\\
 \end{pmatrix},
$$
then the group $\Gmax$ is generated by
$$
 \rho_k
  = \lb \exp \lb 2 \pi \sqrt{-1} \varphi_1^{(k)} \rb, \dots, 
      \exp \lb 2 \pi \sqrt{-1} \varphi_n^{(k)} \rb \rb,
  \qquad k = 1, \dots, n.
$$
Put
$$
 \varphi_i = \varphi_i^{(1)} + \cdots + \varphi_i^{(n)},
  \qquad i = 1, \dots, n
$$
and define a homomorphism
$$
 \varphi : \bCx \to K
$$
by
$$
 \varphi(\alpha) = 
  \lb
   \alpha^{\ell \varphi_1}, \dots, \alpha^{\ell \varphi_n}
  \rb,
$$
where $\ell$ is the smallest integer
such that $\ell \varphi_i \in \bZ$ for $i = 1, \dots, n$.
Then $\varphi$ is injective and one has an exact sequence
$$
 1 \to \bCx \xto{\varphi} K \to \Gmaxbar \to 1,
$$
where $\Gmaxbar := \coker \varphi$.
$W$ is quasi-homogeneous of degree $\ell$
with respect to the $\bZ$-grading
$$
 \deg x_i = \ell \varphi_i, \qquad i = 1, \dots, n.
$$
The intersection
$
 \Image \varphi \cap G
$
is generated by
$$
 J =
 \lb \exp \lb 2 \pi \sqrt{-1} \varphi_1 \rb, \dots, 
      \exp \lb 2 \pi \sqrt{-1} \varphi_n \rb \rb.
$$
Let $G$ be a subgroup of $\Gmax$ containing $J$ and
$\Gbar = G / \la J \ra$ be its image in $\Gmaxbar$.
The inverse image of $\Gbar$ by $K \to \Gmaxbar$ will be denoted by $H$,
and we write the group of characters of $H$ as $M$.
Let $A^\ast$ be the transpose matrix of $A$.
The group $G_{\mathrm{max}}^\ast$ of maximal diagonal symmetries of $W^\ast$
is generated by the column vectors $\rhobar_i$
of $(A^\ast)^{-1} = (A^{-1})^\ast$.
The {\em transpose} $G^\ast$ of the subgroup $G \subset \Gmax$
is defined
by Krawitz \cite{Krawitz} as
$$
 G^\ast = \lc \prod_{i=1}^n \rhobar_i^{r_i} \left|
   \begin{pmatrix}
     r_1 & \cdots & r_n
   \end{pmatrix}
    A^{-1}
   \begin{pmatrix}
     a_1 \\ \vdots \\ a_n
   \end{pmatrix}
    \in \bZ \text{ for all }
   \prod_{i=1}^n \rho_i^{a_i} \in G
 \right. \rc .
$$
The {\em transposition mirror symmetry} of
Berglund and H\"{u}bsch \cite{Berglund-Hubsch}
states that the pairs $(W, G)$ and $(W^\ast, G^\ast)$
are mirror dual to each other.
Homological mirror symmetry is expected to take the form
$$
 D^b_\sing(R) \cong D^b \Fuk^{G^\ast} W^\ast
$$
where $R = \bC[x_1, \dots, x_n] / (W)$ is an $M$-graded ring,
although the ``orbifold Fukaya category'' $\Fuk^{G^\ast} W^\ast$
on the right hand side is not defined yet.

If $W^\ast$ is a polynomial in two variables and
$
 G^\ast \subset \SL_2(\bC),
$
then the map
$$
 W^\ast : \bC^2 \to \bC
$$
descends to the map
$$
 \Wbar^\ast : \bC^2 / G^\ast \to \bC,
$$
which can be pulled-back to the crepant resolution
$
 \pi : Y \to \bC^2 / G^\ast;
$
$$
 \Wtilde^\ast = \Wbar^\ast \circ \pi : Y \to \bC.
$$
One can replace $\Fuk^{G^\ast} W^\ast$
with the Fukaya category $\Fuk \Wtilde^\ast$
of a perturbation of $\Wtilde^\ast$ and
formulate the following conjecture:

\begin{conjecture} \label{conj:transposition_cr}
If $n = 2$ and $G^\ast \subset \SL_2(\bC)$,
then one has an equivalence
$$
 D^b_\sing(R) \cong D^b \Fuk \Wtilde^\ast
$$
of triangulated categories.
\end{conjecture}

As an example,
consider the case when
$$
 W^\ast(u, v) = u^3 v + v^2
$$
is the transpose of a polynomial of type $D_4$ and
$
 G^\ast = \la \frac{1}{2}(1,1) \ra \subset \SL_2(\bC)
$
is a cyclic group of order two
generated by $\diag(-1, -1)$.
The invariant ring $\bC^2[x, y]^{G^\ast}$ is generated by
\begin{align*}
 x &= u^2, \\
 y &= v^2, \\
 z &= u v
\end{align*}
with the relation
$$
 x y = z^2,
$$
and $\Wbar^\ast$ is given by
$$
 \Wbar^\ast = x z + y.
$$
The minimal resolution of
$$
 S = \bC^2 / G = \Spec \bC[x, y, z] / (x y - z^2)
$$
is obtained by a blow-up
along the ideal $(x, z) \subset \bC[x, y, z]$,
which is covered by a chart
\begin{align*}
 x &= x_1, \\
 y &= y, \\
 z &= x_1 z_1,
\end{align*}
with a local coordinate $(x_1, z_1)$,
and another chart
\begin{align*}
 x &= x_2 z_2, \\
 y &= y, \\
 z &= z_2,
\end{align*}
with a local coordinate $(x_2, y)$.
The map $\Wtilde^\ast$ is written as
$$
 \Wtilde^\ast = x_1^2 z_1 + x_1 z_1^2
$$
on the first chart
where it has a $D_4$-singularity at the origin,
and as
$$
 \Wtilde^\ast
 = x_2^3 y^2 + y
$$
on the second chart,
where it does not have any critical point.

The transpose of $\lb W^\ast, \la \frac{1}{2}(1, 1) \ra \rb$ is given by
$
 \lb x^3 + x y^2, \la \frac{1}{3}(1, 1) \ra \rb,
$
where $\la \frac{1}{3}(1, 1) \ra \subset \GL_2(\bC)$ is
the cyclic group of order three generated by
$\diag(\exp(2 \pi \sqrt{-1} / 3), \exp(2 \pi \sqrt{-1} / 3))$.
The corresponding abelian group $M$ is isomorphic to $\bZ$,
and the resulting grading of $R = \bC[x, y] / (x^3 + x y^2)$
is given by $\deg x = \deg y = 1$.
One can show an equivalence
$$
 D^b_\sing(R) \cong D^b \module \Gamma
$$
with the derived category of a Dynkin quiver $\Gamma$ of type $D_4$
just as in Section \ref{sc:dbsing},
and Conjecture \ref{conj:transposition_cr} holds in this case.

{\em Acknowledgment}:
M.~F. is supported by Grant-in-Aid for Young Scientists (No.19.8083).
K.~U. is supported by Grant-in-Aid for Young Scientists (No.18840029).

\bibliographystyle{amsalpha}
\bibliography{bibs}

\noindent
Masahiro Futaki

Graduate School of Mathematical Sciences,
The University of Tokyo,
3-8-1 Komaba Meguro-ku Tokyo 153-8914, Japan

{\em e-mail address}\ : \  futaki@ms.u-tokyo.ac.jp

\ \\

\noindent
Kazushi Ueda

Department of Mathematics,
Graduate School of Science,
Osaka University,
Machikaneyama 1-1,
Toyonaka,
Osaka,
560-0043,
Japan.

{\em e-mail address}\ : \  kazushi@math.sci.osaka-u.ac.jp

\end{document}